\documentclass[psamsfonts]{amsart}

\usepackage{amssymb,amsfonts}
\usepackage{amsmath}
\usepackage[all,arc]{xy}
\usepackage{enumerate}
\usepackage{mathrsfs}
\usepackage{mathtools}
\usepackage{cite}
\usepackage[unicode]{hyperref}
\usepackage{bookmark}

\newtheorem{thm}{Theorem}[section]
\newtheorem{theorem}[thm]{Theorem}
\newtheorem{cor}[thm]{Corollary}

\newtheorem{lemma}[thm]{Lemma}

\newtheorem{fact}[thm]{Fact}

\theoremstyle{definition}
\newtheorem{defn}[thm]{Definition}

\newcommand{\defeq}{\mathrel{\mathop:}=}

\newcommand{\Q}{\mathbb{Q}}
\newcommand{\PP}{\mathbb{P}}

\newcommand{\LL}{\mathbb{L}}

\newcommand{\HH}{\mathbb{H}}
\newcommand{\FF}{\mathbb{F}}

\newcommand{\la}{\lambda}
\newcommand{\ka}{\kappa}
\newcommand{\w}{\omega}

\newcommand{\mc}{\mathcal}

\DeclareMathOperator{\cf}{cf}

\DeclareMathOperator{\Coll}{Col}

\DeclareTextCommand{\textaleph}{L8U}{ℵ}

\newcommand{\forces}{\Vdash}
\newcommand{\cat}{^\smallfrown}

\makeatletter
\let\c@equation\c@thm
\makeatother
\numberwithin{equation}{section}

\title{Tree Properties at Successors of Singulars of Many Cofinalities}
\author{William Adkisson}
\begin{document}

\begin{abstract}
 From many supercompact cardinals, we show that it is consistent for the tree property to hold at many small successors of singular cardinals, each with a different cofinality. In particular, we construct a model in which the tree property holds at $\aleph_{\w+\w+1}$ and at $\aleph_{\w_n+1}$ for all $0<n<\w$. We show that this can be done for the strong tree property as well, and extend the technique to large uncountable sequences of desired cofinalities.
\end{abstract}

	\maketitle

\section{Introduction}

The tree property is a generalization of K\"onig's Lemma to uncountable cardinals, and is closely related to large cardinals. In particular, the tree property holds at an inaccessible cardinal $\ka$ if and only if $\ka$ is weakly compact. The tree property can consistently hold at small cardinals, but it retains some large cardinal strength: Mitchell and Silver \cite{mitchell:tp} showed that the tree property at $\aleph_2$ is equiconsistent with a weakly compact.

In the 70s, Magidor asked whether it was consistent for every regular cardinal greater than $\aleph_1$ to have the tree property. Since then, set theorists have constructed models where the tree property holds at more and more successive cardinals; see for instance \cite{abraham:treepropN2N3}, \cite{CF_TreeProp}, \cite{MS_TPSuccSing}, and \cite{NeemanTPNw+1}.

More recently, there has been a resurgence of interest in strengthenings of the tree property that are linked with more powerful large cardinals. The strong tree property is a generalization of the tree property that holds at an inaccessible cardinal $\ka$ if and only if $\ka$ is strongly compact \cite{jech:combinatorialprobs}. 
Like the tree property, 
this property can consistently hold at small cardinals, and is viewed as strong evidence for the presence of a strongly compact cardinal.

Magidor's question can be extended to these stronger properties, and in the past decade many results about the tree property have been generalized. See for instance \cite{fontanella:CF}, \cite{UngerCF}, \cite{HachtmanITPNw+1}, and \cite{adkisson:ITP}.

If the tree property and its strengthenings are to hold at every regular cardinal, they will need to hold at many successors of singular cardinals. In particular, they will need to hold at successors of singulars of many different cofinalities. It is easy to arrange this situation at the level of large cardinals: Magidor and Shelah \cite{MS_TPSuccSing} proved that the successor of a limit of supercompacts always has the tree property, so a long enough sequence of supercompacts will give rise to many such limits of different cofinalities.
Obtaining this result at small cardinals is much more difficult. In \cite{adkisson:ITP}, the author forced the strong tree property at finitely many successors of singulars with different cofinalities simultaneously. 
Unfortunately, the techniques used in that paper are fundamentally restricted to finite sequences of cofinalities, and cannot be easily generalized.

In this paper we present a different technique that can be used to obtain the tree property and the strong tree property at successors of singulars of infinitely many cofinalities simultaneously. In particular, we will obtain the strong tree property at $\aleph_{\w+\w+1}$ along with $\aleph_{\w_n+1}$ for all $0< n < \w$.

In Section \ref{section:lemmas}, we define the strong tree property and list a number of standard lemmas that we will use. In Section \ref{section:TP} we will obtain the tree property at $\aleph_{\w+\w+1}$ and each $\aleph_{\w_n+1}$. In Section \ref{section:sTP}, we prove the same result for the strong tree property. Section \ref{s:extending} extends these results to uncountably many successors of singulars of different cofinalities.

\section{Preliminaries and Branch Lemmas}\label{section:lemmas}

First, let us define the tree property and strong tree property.

\begin{defn}
	A tree $T$ is a \emph{$\ka$-tree} if it has height $\ka$ and levels of size $<\ka$. A cardinal $\ka$ has the \emph{tree property} if every $\ka$-tree has a cofinal branch.
\end{defn}

The strong tree property is concerned with a more general class of objects, called $\mc{P}_\ka(\la)$-lists.

\begin{defn}
	A sequence $d = \langle d_z \mid z\in \mc{P}_\ka(\la)\rangle$ is a $\mc{P}_\ka(\la)$-list if $d_z\subseteq z$ for all $z \in \mc{P}_\ka(\la)$.
\end{defn}
\begin{defn}
	The $z$-th level of a $\mc{P}_\ka(\la)$-list, written $L_z$, is defined as
	\[L_z \defeq \{d_y\cap z \mid y \supseteq z\}.\]
	A $\mc{P}_\ka(\la)$-list is \emph{thin} if $|L_z|<\ka$ for all $z\in \mc{P}_\ka(\la)$.
\end{defn}
\begin{defn}
	A subset $b\subseteq \la$ is a \emph{cofinal branch} through a $\mc{P}_\ka(\la)$-list $d$ if $b\cap z \in L_z$ for all $z\in \mc{P}_\ka(\la)$.
\end{defn}
\begin{defn}
	The \emph{strong tree property} holds at $\ka$ if for all $\la > \ka$, every thin $\mc{P}_\ka(\la)$-list has a cofinal branch.
\end{defn}
\begin{fact}\cite[Lemma 3.4]{fontanella:stpsuccsing}\label{fact:tpgoesdown}
	Suppose $\la < \la'$. If every thin $\mc{P}_\ka(\la')$-list has a cofinal branch, then so does every thin $\mc{P}_\ka(\la)$-list.
\end{fact}

In the specific case where $\ka = \la$, it is often convenient to restrict to elements of $\ka$ rather than of $\mc{P}_\ka(\ka)$.

\begin{defn}
	Let $\ka$ be a regular cardinal. A sequence $d = \langle d_\alpha \mid \alpha < \ka\rangle$ is a \emph{$\ka$-list} if $d_\alpha \subseteq \alpha$ for all $\alpha < \ka$. The \emph{$\alpha$-th level}, denoted $L_\alpha$, is the set $\{d_\beta \cap \alpha \mid \alpha < \beta < \ka\}$. A $\ka$-list is \emph{thin} if every level has size $<\ka$. A set $b \subseteq \ka$ is an \emph{cofinal branch} through a $\kappa$-list $d$ if $b\cap \alpha \in L_\alpha$ for all $\alpha < \ka$.
\end{defn}

Since $\ka$ is unbounded in $\mc{P}_\ka(\ka)$, we have the following easy facts:
\begin{fact}\label{fact:ka-lists(ka,ka)-lists}
If every thin $\ka$-list has a cofinal branch, then every thin $\mc{P}_\ka(\ka)$-list has a cofinal branch.
\end{fact}
\begin{fact}\label{fact:ka-lists2}
	Let $d = \langle d_z\mid z\in \mc{P}_\ka(\ka)\rangle$ be a thin $\mc{P}_{\ka}(\ka)$-list. Let $d^*$ be the restriction of $d$ to elements indexed by ordinals; that is, $d^* = \langle d_\alpha \mid \alpha < \ka\rangle$. Then $d^*$ is a thin $\ka$-list, and if $d^*$ has a cofinal branch, then so does $d$.
\end{fact}

%

We will make heavy use of the following general lemma, a slight weakening of \cite[Theorem 3.10]{NeemanTPNw+1}, for obtaining the tree property at successors of singulars.


\begin{lemma}\label{lem:neemantp}
	Let $\tau$ be a regular cardinal. Let $\langle \ka_\rho \mid \rho < \tau\rangle$ be a continuous increasing sequence of regular cardinals above $\tau$, with supremum $\nu$. Let $I \subseteq \ka_0$, and fix $\rho' < \tau$. For each $\mu \in I$, let $\LL_\mu$ be a forcing poset of size $\leq \ka_{\rho'}$. Suppose the following hold:
	\begin{enumerate}
		\item For each $\mu \in I$, $\LL_\mu$ is the product of forcings $\PP_\mu$ and $\Q_\mu$, where $|\PP_\mu|< \mu^+$ and $\Q_\mu$ is $\mu^+$-closed.
		\item $\ka_0$ is $\nu^+$-supercompact, with a normal measure $U_0$ on $\mc{P}_{\ka_0}(\nu^+)$ and a corresponding embedding $i$ such that $\nu \in i(I)$.
		\item For all ordinals $\rho < \tau$ there is a generic $\nu^+$-supercompactness embedding $j_{\rho+2}$ with domain $V$ and critical point $\ka_{\rho+2}$, added by a poset $\FF$ whose full support power $\FF^{\ka_\rho}$ is $<\ka_\rho$-distributive in $V$.
	\end{enumerate}
	Then there exists $\mu \in I$ such that in the extension of $V$ by $\LL_\mu$, the tree property holds at $\nu^+$.
\end{lemma}
Note that given any thin $\nu^+$-list $d$, we can build a corresponding $\nu^+$-tree: nodes are all sets of the form $d_\beta\cap \alpha$ for $\alpha \leq \beta< \nu^+$, and the tree relation is inclusion. The levels of this tree are precisely the levels of the original list, and any cofinal branch through this tree is also a cofinal branch through the list. So we immediately obtain the following lemma, which we will use to obtain cofinal branches through thin lists. (In fact, the conclusions of the two lemmas are equivalent.)
\begin{lemma}\label{lem:strtp}
	Under the same hypotheses as the previous lemma, there is $\mu \in I$ such that in the extension of $V$ by $\LL_\mu$, every thin $\nu^+$-list has a cofinal branch.
\end{lemma}

A key tool used to work with cofinal branches through trees and thin lists is the (thin) $\ka$-approximation property.

\begin{defn}\label{def:approxed}
	Let $\ka$ be regular, $\la$ be an ordinal, and $\PP$ be a forcing notion in a model $V$. A $\PP$-name $\dot{b}$ for a subset of $\la$ is \emph{$\ka$-approximated by $\PP$ over $V$} if for all $z \in (\mc{P}_{\ka}(\la))^V$, $\forces_\PP \dot{b}\cap z \in V$.
	
	A $\PP$-name $\dot{b}$ for a subset of $\la$ is \emph{thinly $\ka$-approximated by $\PP$ over $V$} if it is $\ka$-approximated by $\PP$ over $V$, and furthermore for every $z \in (\mc{P}_{\ka}(\la))^V$, $|\{x \in V \mid \exists p\in \PP \ p \forces_\PP x = \dot{b}\cap z\}|<\ka$.
\end{defn}

\begin{defn}\label{def:ka-approx}
	Let $\ka$ be regular. A forcing $\PP$ has the \emph{$\ka$-approximation property} over a model $V$ if for every ordinal $\la$ and $\PP$-name $\dot{b}$ for a subset of $\la$, if $\dot{b}$ is $\ka$-approximated by $\PP$ over $V$, then $\forces_{\PP}\dot{b} \in V$.
	
	A forcing $\PP$ has the \emph{thin $\ka$-approximation property} over $V$ if for every ordinal $\la$ and every $\PP$-name $\dot{b}$ for a subset of $\la$, if $\dot{b}$ is thinly $\ka$-approximated by $\PP$ over $V$, then $\forces_{\PP}\dot{b} \in V$.
\end{defn}

The $\ka$-approximation property implies the thin $\ka$-approximation property. 

A cofinal branch through a thin $\mc{P}_\ka(\la)$ list is always thinly approximated over $V$, by any forcing. A similar argument shows that branches through a $\ka$-tree are thinly approximated by any forcing; see \cite[Section 2]{UngerAtrees} for details.

\begin{lemma}\label{lem:branchapprox}
	Let $d$ be a thin $\mc{P}_\ka(\la)$ list in $V$, and let $\mathbb{P}$ be a notion of forcing over $V$. Suppose $\dot{b}$ is a $\PP$-name for a cofinal branch through this list. Then $\dot{b}$ is thinly $\ka$-approximated by $\PP$ over $V$.
\end{lemma}
\begin{proof}
	Let $\dot{b}$ be a name for an cofinal branch, and let $z \in (\mc{P}_\ka(\la))^V$. Since $\dot{b}$ is forced to be cofinal, it must meet every level; that is, the empty condition forces that $\dot{b} \cap z \in L_z$, for all $z \in \mc{P}_\ka(\la)$. Since $L_z$ is in $V$, $\dot{b}\cap z$ is likewise forced to be in $V$. Since the list is thin, $|L_z| < \ka$, and so there are fewer than $\ka$ possibilities for $\dot{b}\cap z$.
\end{proof}

Many standard branch lemmas for trees can be generalized to use the (thin) approximation property. We will use of following:

\begin{lemma}\label{lem:approxprop}\cite[Lemma 2.4]{UngerAtrees}
	Let $\ka$ be a regular cardinal. Suppose that $\PP$ is a forcing poset such that $\PP\times \PP$ is $\ka$-cc. Then $\PP$ has the $\ka$-approximation property.
\end{lemma}
\begin{cor}\label{cor:smallbranch}
	Let $\ka$ be a regular cardinal. Suppose that $\PP$ is a forcing poset with $|\PP| < \ka$. Then $\PP$ has the $\ka$-approximation property.
\end{cor}

In particular, we will make key use of the following lemma, which generalizes  \cite[Theorem 2.1]{MS_TPSuccSing}. The original lemma shows that these forcings can't add branches to trees; the approximation property is needed to obtain the strong tree property in Section \ref{section:sTP}.

\begin{lemma}\label{lem:cc+closed=approx}\cite[Lemma 2.20]{adkisson:ITP}
	Suppose $\nu$ is a singular strong limit cardinal with cofinality $\tau$. Let $\Q$ be a $\mu^+$-closed forcing over a model $V$ for some $\mu<\nu$ with $\tau \leq \mu$, and let $\PP\in V$ be a forcing poset with $|\PP|\leq \mu$. Then $\Q$ has the thin $\nu^+$-approximation property in the generic extension of $V$ by $\PP$.
\end{lemma}

Finally, our argument will make crucial use of the following standard absorption lemma.
\begin{lemma}\label{lem:colabsorption}\cite[Theorem 14.3]{cummings:handbook}
	Let $\ka$ be an inaccessible cardinal, and let $\delta < \ka$ be regular. Let $\PP$ be a $\delta$-closed forcing poset with $|\PP| < \ka$. Then there is a forcing projection from $\Coll(\delta, <\ka)$ to $\PP$ whose quotient is $\Coll(\delta, <\ka)$.
\end{lemma}

\section{The Tree Property at Countably Many Cofinalities}\label{section:TP}
First we will show that the tree property can be forced to hold at $\aleph_{\w+\w+1}$ and at each $\aleph_{\w_n+1}$ simultaneously for $0<n<\w$. To obtain the tree property at each $\aleph_{\w_n+1}$ using Lemma \ref{lem:neemantp}, we need to choose a singular cardinal $\mu_n$ whose successor will become $\w_{n+1}$ in the final model. The cofinality of $\mu_n$ needs to be the cardinal that will become $\w_{n}$ in final model; that is, $\cf(\mu_n) = \mu_{n-1}^+$, and $\mu_n$ will be collapsed to $\mu_{n-1}^+$. Once all of these collapses are complete, $\aleph_{\w}$ will be the supremum of the sequence $\langle \mu_n^+\mid n < \w\rangle$. Since each $\mu_n^+$ is the successor of a singular, there is no reason to expect the tree property to hold at $\aleph_{\w+1}$ in this construction; we settle for the next best cardinal, $\aleph_{\w+\w+1}$.

The primary difficulty is in selecting the sequence $\langle \mu_i \mid i<\w\rangle$. If we choose them one at a time, we run into difficulties. In particular, when selecting each $\mu_i$ we need complete knowledge of the poset (so all $\mu_j$ for $j\neq i$) to ensure that we will have the tree property at the appropriate cardinals. If we select a different $\mu_i$ for every possible tail of the sequence, we are unable to combine all of our choices at the limit stage, since we will never make a selection that does not depend on other information. 
Working within these constraints, we can obtain some partial results as in \cite{adkisson:ITP}, but those techniques only work for finitely many cofinalities.

Our solution is to obtain the tree property in a larger poset that does not depend on any future choices. This poset will project down to our target model, regardless of the selections we make, and the quotient of this projection will not add branches through trees. This approach is loosely inspired by the arguments of Golshani and Hayut in \cite{golshani-hayut:tpcountablesegment}, although the details of the construction are quite different.

\begin{theorem}
	Let $\langle \ka_\alpha \mid \alpha < \ka_0\rangle$ be an continuous increasing sequence of cardinals with the following properties:
	\begin{itemize}
			\item $\ka_{\alpha+1} = \ka_\alpha^+$ for all limit ordinals $\alpha$
			\item $\ka_n$ is indestructibly supercompact for all $n<\w$
			\item $\ka_{\alpha+2}$ is indestructibly supercompact for all $\w\leq \alpha < \ka_0$.
		\end{itemize}

	Then there is a generic extension in which the tree property holds at $\aleph_{\w+\w+1}$ and at $\aleph_{\w_n+1}$ for all $0<n<\w$.
\end{theorem}
\begin{proof}
	Define the poset $\HH$ as follows:
	\[\HH\defeq \left(\prod_{n<\w} \Coll(\ka_n, <\ka_{n+1})\right) \times \left( \prod_{\w\leq\rho<\tau} \Coll(\ka_{\rho+1}, <\ka_{\rho+2})\right).\] Let $H$ be generic for $\HH$, and work in $V[H]$. Note that in $V[H]$, requirements (2) and (3) of Lemma \ref{lem:neemantp} are satisfied.
	
	Let $I$ be the collection of all strictly increasing sequences $\langle \mu_i \mid i<\w\rangle$, where $\mu_i <\ka_0$ is a singular cardinal of cofinality $\mu_{i-1}^+$. For each sequence $s = \langle \mu_i \mid i<\w\rangle \in I$, define 
	\[\LL_s \defeq \Coll(\w, \mu_0) \times \left(\prod_{n<\w} \Coll(\mu_{n}^+, \mu_{n+1})\right)\times \Coll\left((\sup_{n<\w}\mu_n)^+, <\ka_0\right).\]
	
	Given such a sequence $s$, for all $i$ with $0<i<\w$, define $\nu_i(s) \defeq \sup_{\rho < \mu_{i-1}^+} \ka_\rho$. Let $\nu_0(s) = \ka_0^{+\w}$. In the extension of $V[H]$ by $\LL_s$, $\nu_i(s)$ will become $\aleph_{\w+\w_i}$. Note that $\nu_i(s)$ only depends on the initial segment $\langle \nu_n \mid n<i\rangle$; when these values are fixed, as they will be throughout our argument, we will usually omit the parameter $s$.
	
	Our goal is to inductively build a sequence $s = \langle \mu^*_k \mid k<\w\rangle$ such that in the extension of $V[H]$ by a generic for $\LL_s$, the tree property will hold at every $\nu_{i}^+(s)$. We will do so inductively, selecting each $\mu_k^*$ based on the previous choices.

	\begin{lemma}\label{lem:indtp}
		Let $k < \w$, and fix $\langle \mu^*_i \mid i<k\rangle$. There exists $\mu_k^*$ such that for all sequences $s = \langle \mu_i \mid i<\w\rangle$ with $\mu_i = \mu_i^*$ for $i \leq k$, the tree property holds at $\nu_k^+(s)$ in $V[H][L_s]$.
	\end{lemma}
	\begin{proof}
		If $k = 0$, for all $\mu_0 <\ka_0$, we set
		\[\PP_0^*(\mu_0) \defeq \Coll(\w, \mu_0).\]
		Otherwise, for all $\mu_i$ with $\mu^*_i < \mu_i < \ka_0$, define
		\[\PP^*_k(\mu_k) \defeq \Coll(\w, \mu_0^*) \times \left(\prod_{n<k-1}\Coll((\mu^*_n)^+, \mu^*_{n+1})\right)\times \Coll((\mu^*_{k-1})^+, \mu_k).\]
		Define
		\[\Q^*_k(\mu_k) \defeq \Coll(\mu_k^+, <\ka_0) \times \left(\prod_{\mu_k < \alpha < \ka_0} \Coll(\alpha^+, <\ka_0)\right),\]
		where the product has $\mu_k^{++}$-support. Finally, we define
		\[\LL^*_k(\mu_k) \defeq \PP^*_k(\mu_k) \times \Q^*_k(\mu_k).\]
		Note that $|\PP^*_{k}(\mu_k)| = \mu_k$, and $\Q^*_k(\mu_k)$ is $\mu_k^+$-closed; $|\LL^*_k(\mu_k)| < \ka_0^{++} = \ka_2$. Thus it meets the first hypothesis of Lemma \ref{lem:neemantp}.
		
		We apply Lemma \ref{lem:neemantp} to conclude that there is some $\mu^*_k$ such that in the extension of $V[H]$ by a generic $L$ for $\LL^*_k(\mu^*_k)$, the tree property holds at $\nu_k^+$.
		
		Now, let $s = \langle \mu_i \mid i<\w\rangle$ with $\mu_n = \mu_n^*$ for $n \leq k$. In $V[H]$, let $\dot{T}$ be a $\LL_s$-name for a $\nu_k^+$-tree.
		
		We wish to show that $\LL^*_k(\mu^*_k)$ projects to $\LL_s$. Note that $\LL_s = \PP_s \times \Q_s$, where
		\[\PP_s = \Coll(\w, \mu_0) \times \left(\prod_{n<k} \Coll(\mu_{n}^+, \mu_{n+1})\right)\]
		and
		\[\Q_s = \left(\prod_{k\leq n<\w} \Coll(\mu_{n}^+, \mu_{i+1})\right)\times \Coll\left((\sup_{i<\w}\mu_i)^+, <\ka_0\right).\]
		Note that since $\mu^*_n = \mu_n$ for $n \leq k$, $\PP_s = \PP^*_k(\mu^*_k)$.
		
		By Lemma \ref{lem:colabsorption}, noting that $\prod_{k<n<\w} \Coll(\mu_{n}^+, \mu_{i+1})$ is $(\mu^*_k)^+$-closed and has size less than $\ka_0$, the first component of $\Q^*_k(\mu^*_k)$ projects onto the first component of $\Q_s$. The second component of $\Q_k^*(\mu^*_k)$ projects onto the second component of $\Q_s$ by restricting to the coordinate indexed by $\alpha = \sup_{i<\w}\mu_i$.
		
		Let $L_s$ be generic for $\LL_s$ over $V[H]$, and let $L_k^*(\mu_k^*)$ be a generic for $\LL^*_k(\mu^*_k)$ over $V[H]$ projecting to $L_s$. Let $T$ be the interpretation of $\dot{T}$ by this generic. Then $T \in V[H][L_s]$. Being a $\nu_k^+$-tree is upwards absolute, so $T$ remains a tree in $V[H][L^*_k(\mu^*_k)]$. As before, since the tree property holds in this model, there must be a cofinal branch through $T$ in $V[H][L^*_k(\mu^*_k)]$. The quotient has size $<\nu_k^+$, so it cannot have added the branch. We conclude that $T$ must have a cofinal branch in $V[H][L_s]$.
	\end{proof}

	Let $s = \langle \mu_i^* \mid i < \w\rangle$. Let $L_s$ be generic for $\LL_s$. Then for each $i < \w$, by Lemma \ref{lem:indtp}, we conclude that in $V[H][L_s]$, the tree property holds at $\nu_i^+(s)$. In this extension, $\nu_i^+(s) = \aleph_{\w+\w_i+1}$, so we have the desired result.
\end{proof}

\section{The Strong Tree Property at Countably Many Cofinalities.}\label{section:sTP}
In this section we extend our construction to force the strong tree property at $\aleph_{\w+\w+1}$ and at each $\aleph_{\w_n+1}$ simultaneously for $n<\w$. When we attempt to apply the techniques of the previous section, we run into a new obstacle: thin $\mc{P}_{\nu^+}(\la)$-lists are not necessarily upwards absolute. When we build our larger model where the strong tree property will hold, we add new small subsets of $\la$. The list we wish to find a branch through may not be a $\mc{P}_\ka(\la)$-list in the larger model, so we can't apply the strong tree property to find the branch. To get around this, we use an auxiliary forcing that collapses $\la$ to $\nu^+$. This lets us work with a single cardinal parameter, reducing the problem to finding a branch through a $\nu^+$-list. Since thin $\nu^+$-lists are upwards absolute, we can use the techniques of the previous section.

\begin{theorem}
	Let $\langle \ka_\alpha \mid \alpha < \ka_0\rangle$ be an continuous increasing sequence of cardinals with the following properties:
	\begin{itemize}
		\item $\ka_{\alpha+1} = \ka_\alpha^+$ for all limit ordinals $\alpha$
		\item $\ka_n$ is indestructibly supercompact for all $n<\w$
		\item $\ka_{\alpha+2}$ is indestructibly supercompact for all $\w\leq \alpha < \ka_0$.
	\end{itemize}
	
	Then there is a generic extension in which the strong tree property holds at $\aleph_{\w+\w+1}$ and at $\aleph_{\w_n+1}$ for all $0<n<\w$.
\end{theorem}
\begin{proof}
	
	Define the poset $\HH$ as before:
	\[\HH\defeq \left(\prod_{n<\w} \Coll(\ka_n, <\ka_{n+1})\right) \times \left( \prod_{\w\leq\rho<\ka_0} \Coll(\ka_{\rho+1}, <\ka_{\rho+2})\right).\] Let $H$ be generic for $\HH$. As before, $V[H]$ meets the requirements given in Lemma \ref{lem:strtp}.
	
	Let $I$ be the collection of all strictly increasing sequences $\langle \mu_i \mid i<\w\rangle$, where $\mu_i <\ka_0$ is a singular cardinal of cofinality $\mu_{i-1}^+$. For each sequence $s = \langle \mu_i \mid i<\w\rangle \in I$, define 
	\[\LL_s \defeq \Coll(\w, \mu_0) \times \left(\prod_{n<\w} \Coll(\mu_{n}^+, \mu_{n+1})\right)\times \Coll\left((\sup_{n<\w}\mu_n)^+, <\ka_0\right).\]
	
	Given such a sequence $s$, for all $i$ with $0<i<\w$, define $\nu_i(s) \defeq \sup_{\rho < \mu_{i-1}^+} \ka_\rho$. As before, we will freely omit the parameter $s$. Let $\nu_0 = \ka_0^{+\w}$. In the extension of $V[H]$ by $\LL_s$, $\nu_i$ will become $\aleph_{\w+\w_i}$.
	
	As before, we inductively build a sequence $s$ such that in the extension of $V[H]$ by a generic for $\LL_s$, the strong tree property will hold at every $\nu_{i}^+$. To do this, we will need to make use of an auxiliary collapse.

	\begin{lemma}\label{lem:indstp}
		Let $k<\w$, and fix $\langle \mu^*_i \mid i<k\rangle$. There exists $\mu_k^*$ such that for all sequences $s = \langle \mu_i \mid i<\w\rangle$ with $\mu_i = \mu_i^*$ for $i \leq k$, the strong tree property holds at $\nu_k^+$ in $V[H][L_s]$.
	\end{lemma}
	\begin{proof}
		Fix $\la \geq \nu_k^+$. Let $K_\la$ be generic for $\Coll(\nu_k, \la)^{V[H]}$. Note that in $V[H][K_\la]$, $\la$ is collapsed to an ordinal with cardinality and cofinality $\nu_k^+$.
		
		If $k = 0$, for all $\mu_0 < \ka_0$, define
		\[\PP_0^*(\mu_0) \defeq \Coll(\w, \mu_0).\]
		Otherwise, for all $\mu_k$ with $\mu^*_{k-1} < \mu_k < \ka_0$, in $V[H]$ define
		\[\PP^*_k(\mu_k) \defeq \Coll(\w, \mu_0^*) \times \left(\prod_{n<k-1}\Coll((\mu^*_n)^+, \mu^*_{n+1})\right)\times \Coll((\mu^*_{k-1})^+, \mu_k).\]
		Finally, define
		\[\Q^*_k(\mu_k) \defeq \Coll(\mu_k^+, <\ka_0) \times \left(\prod_{\mu_k < \alpha < \ka_0} \Coll(\alpha^+, <\ka_0)\right),\]
		where the product has $\mu_k^{++}$-support, and
		\[\LL^*_k(\mu_k) \defeq \PP^*_k(\mu_k) \times \Q^*_k(\mu_k).\]
		Note that $|\PP^*_{k}(\mu_k)| = \mu_k$, $\Q^*_k(\mu_k)$ is $\mu_k^+$-closed, and $|\LL^*_k(\mu_k)| < \ka_0^{++} = \ka_2$. So $\LL^*_k(\mu_k)$ meets the hypotheses of Lemma \ref{lem:strtp}, applied in $V[H][K_\la]$.
		
		We apply Lemma \ref{lem:strtp} to conclude that there is some $\mu^*_k(\la)$ such that in the extension of $V[H][K_\la]$ by a generic $L$ for $\LL^*_k(\mu^*_k(\la))$, every thin $\nu_k^+$-list has a cofinal branch.
		
		In $V[H]$, let $s = \langle \mu_i \mid i<\w\rangle$ with $\mu_n = \mu_n^*$ for $n < k$ and $\mu_k = \mu_k^*(\la)$, and let $\dot{d}$ be a $\LL_s$-name for a thin $\mc{P}_{\nu_k^+}(\la)$-list. Let $L_s$ be generic for $\LL_s$. Once again, $\mc{P}_{\nu_k^+}(\la)^{V[H][K_\la]}=\mc{P}_{\nu_k^+}(\la)^{V[H]}$, and $\la$ is collapsed to $\nu_k^+$. By Fact\ref{fact:ka-lists2}, we can build a thin $\ka$-list $d^*$ such that $d'$ has a cofinal branch if $d^*$ does. Since $d'$ is isomorphic to $d$, we conclude that if $d^*$ has a cofinal branch, so does $d$.
		
		Working in $V[H][K_\la]$, we wish to show that $\LL^*_k(\mu^*_k(\la))$ projects to $\LL_s$. Note that $\LL_s = \PP_s \times \Q_s$, where
		\[\PP_s = \Coll(\w, \mu_0) \times \left(\prod_{n<k} \Coll(\mu_{n}^+, \mu_{n+1})\right)\]
		and
		\[\Q_s = \left(\prod_{k<n<\w} \Coll(\mu_{n}^+, \mu_{i+1})\right)\times \Coll\left((\sup_{i<\w}\mu_i)^+, <\ka_0\right).\]
		
		Note that since $\mu^*_n = \mu_n$ for $n < k$ and $\mu^*_k(\la) = \mu_k$, $\PP_{s} = \PP^*_k(\mu^*_k)$.
		By Lemma \ref{lem:colabsorption}, noting that $\prod_{k<n<\w} \Coll(\mu_{n}^+, \mu_{i+1})$ is $(\mu^*_k(\la))^+$-closed and has size less than $\ka_0$, the first component of $\Q^*_k(\mu^*_k(\la))$ projects onto the first component of $\Q_s$. The second component of $\Q_k^*(\mu^*_k(\la))$ projects onto the second component of $\Q_s$ by restricting to the coordinate indexed by $\alpha = \sup_{i<\w}\mu_i$.
		
		Let $L_k^*(\mu_k^*)$ be a generic for $\LL_k^*(\mu_k^*)$ projecting to $L_s$. Since the levels of $d^*$ are indexed by ordinals and $\nu^+_k$ is preserved, $d^*$ remains a thin $\nu_k^+$-tree in $V[H][L^*_k(\mu^*_k(\la))]$. Every thin $\nu_k^+$-list has a cofinal branch in this model, so there must be a cofinal branch through $d^*$ in $V[H][L^*_k(\mu_k(\la))]$. Since the quotient has size $<\nu_k^+$, it cannot have added the branch by Lemma \ref{lem:approxprop}, so $d^*$ must have a cofinal branch in $V[H][K_\la][L_s]$.
		
		Since we have found such a cofinal branch for $d^*$, we conclude that $d$ has a cofinal branch $b$ in $V[H][K_\la][L_s]$. Since $\LL_s$ has size less than $\nu$, and $\Coll(\nu^+, \la)^{V[H]}$ is $\nu^+$-closed in $V[H]$, applying Lemma \ref{lem:cc+closed=approx} we conclude that $\Coll(\nu^+, \la)^{V[H]}$ could not have added a branch through $d$ over $V[H][L_s]$.
		
		It follows that for any sequence $s\in I$ starting with $\langle \mu^*_n \mid n < k\rangle \cat \mu^*_k(\la)$, any thin $\mc{P}_{\nu_k^+(s)}(\la)$-list in $V[H][L_s]$ has a cofinal branch.
		
		Since there are at most $\ka_0$ many options for $\mu^*_k(\la)$, there is a fixed $\mu^*_k$ such that for unboundedly many $\la$, $\mu^*_k(\la) = \mu^*_k$. We conclude (using Fact \ref{fact:tpgoesdown}) that for any sequence $s\in I$ starting with with $\langle \mu^*_n \mid n < k\rangle \cat \mu^*_k(\la)$, the strong tree property holds at $\nu_k^+(s)$ in $V[H][L_s]$.
	\end{proof}
	
	Let $s = \langle \mu_i^* \mid i < \w\rangle$, and let $L_s$ be generic for $\LL_s$. As before, applying Lemma \ref{lem:indtp} inductively, we conclude that in $V[H][L_s]$, the strong tree property holds at $\nu_i^+ = \aleph_{\w+\w_i+1}$ for all $i<\w$.
\end{proof}

\section{Extending Further}\label{s:extending}

The arguments in the previous two sections are not limited to countable sequences. The primary restriction on the number of cofinalities is that the sequence of $\mu_i$'s needs to stay below $\ka_0$. In addition, if we are attempting to obtain the tree property at $\aleph_{\w_\alpha+1}$, our argument requires that $\aleph_{\w_\alpha} > \w_\alpha$. Since we need our target cofinalities to be regular, we avoid this problem by restricting to successor cardinals.

Note however that the more cofinalities we wish to include, the larger our starting point will be. If our desired cofinalities are the first $\tau$-many successor cardinals for some fixed $\tau < \ka_0$, we will select a new $\w_{\alpha+1}$ for all $\alpha < \tau$, and $\ka_0$ will be collapsed to a finite successor of $\aleph_{\tau}$. Then we will obtain the strong tree property at cardinals of the form $\aleph_{\tau+\w_\alpha+1}$. 

\begin{thm}
Let $\langle \ka_\alpha \mid \alpha < \ka_0\rangle$ be an continuous increasing sequence of cardinals with the following properties:
\begin{itemize}
	\item $\ka_{\alpha+1} = \ka_\alpha^+$ for all limit ordinals $\alpha$
	\item $\ka_n$ is indestructibly supercompact for all $n<\w$
	\item $\ka_{\alpha+2}$ is indestructibly supercompact for all $\w\leq \alpha < \ka_0$.
\end{itemize}

Fix $\tau < \ka_0$. Then there is a generic extension in which the strong tree property holds at $\aleph_{\tau+\w_\alpha+1}$ for all successor ordinals $\alpha < \tau$.
\end{thm}
\begin{proof}
	Our construction is as before.
	
	Define the poset $\HH$ as follows:
	\[\HH\defeq \left(\prod_{n<\w} \Coll(\ka_n, <\ka_{n+1})\right) \times \left( \prod_{\w\leq\rho<\ka_0} \Coll(\ka_{\rho+1}, <\ka_{\rho+2})\right).\]
	
	Let $I$ be the collection of all continuous strictly increasing sequences $\langle \mu_{\alpha} \mid \alpha < \tau\rangle$. For all sequences $s = \langle \mu_{\alpha} \mid \alpha < \tau\rangle\in I$, define:

	\[\LL_s \defeq \Coll(\w, \mu_0) \times \left( \prod_{\alpha < \tau } \Coll(\mu_{\alpha}^+, \mu_{\alpha+1})\right) \times \Coll((\sup_{\alpha < \tau} \mu_{\alpha+1})^+,<\ka_0).\]
	
	Note that after forcing with $\LL_s$, $\mu_\alpha^+$ will become $\aleph_{\w_\alpha+1}$.
	We inductively select $\mu^*_\alpha$ such that for any sequence $s\in I$ beginning with $\langle \mu^*_\beta \mid \beta \leq \alpha\rangle$, the strong tree property will hold at $\aleph_{\tau+\alpha+1}$ in $V[H][L_s]$.
	The base case is more or less identical to the base case in Lemma \ref{lem:indstp}, using 
	\[\LL^*_0(\mu_0) = \Coll(\w,\mu_0)\times \Coll(\mu_0^+, <\ka_0) \times \left(\prod_{\mu_0<\beta < \ka_0} \Coll(\beta^+, <\ka_0)\right).\]
	At limit stages $\gamma$, since the sequence must be continuous, we set $\mu^*_\gamma = \sup_{\alpha < \gamma} \mu^*_{\alpha}$.
	At successor stages, we follow the argument of Lemma \ref{lem:indstp}, using
	\[\PP^*_{\alpha+1}(\mu_{\alpha+1}) \defeq \Coll(\w, \mu_0^*)\times \left(\prod_{\beta < \alpha} \Coll((\mu^*_{\beta})^+, \mu^*_{\beta+1})\right)\times\Coll(\mu^*_{\alpha}, \mu_{\alpha+1}),\]
	\[\Q^*_{\alpha+1}(\mu_{\alpha+1}) \defeq \Coll(\mu_{\alpha+1}^+, <\ka_0) \times \left(\prod_{\mu_{\alpha+1} < \beta < \ka_0}\Coll(\beta^+, <\ka_0)\right),\]
	and
	\[\LL_{\alpha+1}^*(\mu_{\alpha+1})\defeq \PP^*_{\alpha+1}(\mu_{\alpha+1}) \times \Q^*_{\alpha+1}(\mu_{\alpha+1}).\]
	When this induction concludes, we will have a sequence $s\in I$ such that in $V[H][L_s]$, the strong tree property holds at $\aleph_{\tau+\w_{\alpha}+1}$ for all successor ordinals $\alpha < \tau$.
\end{proof}

\bibliography{bib}
\bibliographystyle{plain}

\end{document}